\documentclass [14pt]{amsart}
\usepackage{amsmath}
\usepackage{amssymb}
\usepackage{amscd}
\usepackage{epsfig}
\usepackage{amsxtra}
\usepackage{pifont}
\usepackage{mathabx}
\usepackage{MnSymbol}
\usepackage{accents}
\usepackage{scalerel,stackengine}
\usepackage{wasysym}
\usepackage{xcolor}

\stackMath
\newcommand\reallywidehat[1]{%
\savestack{\tmpbox}{\stretchto{%
  \scaleto{%
    \scalerel*[\widthof{\ensuremath{#1}}]{\kern-.6pt\bigwedge\kern-.6pt}%
    {\rule[-\textheight/2]{1ex}{\textheight}}
  }{\textheight}%
}{0.5ex}}%
\stackon[1pt]{#1}{\tmpbox}%
}

\newcommand\reallywidecheck[1]{%
\savestack{\tmpbox}{\stretchto{%
  \scaleto{%
    \scalerel*[\widthof{\ensuremath{#1}}]{\kern-.6pt\bigwedge\kern-.6pt}%
    {\rule[-\textheight/2]{1ex}{\textheight}}
  }{\textheight}%
}{0.5ex}}%
\stackon[1pt]{#1}{\scalebox{-1}{\tmpbox}}%
}

\numberwithin{equation}{section}

\newcommand{\dens}{\mbox{\rm dens}}

\newcommand{\RR}{{\mathbb R}}

\newcommand{\ZZ}{{\mathbb Z}}
\newcommand{\CC}{{\mathbb C}}

\newcommand{\XX}{\mathbb X}
\newcommand{\YY}{\mathbb Y}
\newcommand{\vol}{\mbox{vol}}

\newcommand{\cL}{{\mathcal L}}

\newcommand{\oplam}{\mbox{\Large $\curlywedge$}}

\newcommand{\FB}{{\mathcal FB}}
\newcommand{\cM}{{\mathcal M}}

\newcommand{\cA}{{\mathcal A}}
\newcommand{\cB}{{\mathcal B}}
\newcommand{\cC}{{\mathcal C}}

\newcommand{\dd}{\mbox{\rm d}}
\newcommand{\eps}{\varepsilon}

\newcommand{\Cu}{C_{\mathsf{u}}}
\newcommand{\Cc}{C_{\mathsf{c}}}

\newcommand{\WAP}{\mathcal{W}\hspace*{-1pt}\mathcal{AP}}
\newcommand{\SAP}{\mathcal{S}\hspace*{-2pt}\mathcal{AP}}
\newcommand{\MAP}{\mathcal{M}\hspace*{-1pt}\mathcal{AP}}

\newcommand{\Bap}{\mathcal{B}\hspace*{-1pt}{\mathsf{ap}}}
\newcommand{\Wap}{\mathcal{W}\hspace*{-1pt}{\mathsf{ap}}}
\newcommand{\bap}{Bap}

 \newtheorem{theorem}{Theorem}[section]
 \newtheorem{lemma}[theorem]{Lemma}
 \newtheorem{prop}[theorem]{Proposition}
 \newtheorem{coro}[theorem]{Corollary}

 \newtheorem{definition}[theorem]{Definition}
 \newtheorem{example}[theorem]{Example}
  \newtheorem{remark}[theorem]{Remark}

\newcommand{\ltebe}{\text{\textlquill}}
\newcommand{\rtebe}{\text{\textrquill}}
\newcommand{\rtebeA}{\text{\textrquill}_{\cA}}
\newcommand{\ebeA}{\circledast_{\cA}}
\allowdisplaybreaks

\begin{document}

\title[Orthogonality]{On the orthogonality of measures of different spectral type with
respect to twisted Eberlein convolution}

\author{Nicolae Strungaru}
\address{Department of Mathematical Sciences, MacEwan University \\
10700 -- 104 Avenue, Edmonton, AB, T5J 4S2\\
Phone: 780-633-3440 \\
and \\
Institute of Mathematics ``Simon Stoilow''\\
Bucharest, Romania}
\email{strungarun@macewan.ca}
\urladdr{http://academic.macewan.ca/strungarun/}

\begin{abstract} In this paper we show that under suitable conditions on their Fourier--Bohr coefficients, the twisted Eberlein convolution of a measure with pure point diffraction spectra and a measure with continuous diffraction spectra is zero. In particular, the diffraction spectrum of a linear combinations of the two measures is simply the linear combinations of the two diffraction spectra with absolute value square coefficients.
\end{abstract}



\maketitle

\section{Introduction}

Born from the discovery of quasicrystals \cite{She}, the field of Aperiodic Order is focused on the study of objects which show long range order, typically in the form of a large Bragg diffraction spectrum, but lack translational symmetry.

As introduced by Hof \cite{HOF3}, mathematical diffraction is defined as follows. Given an object, typically a point-set $\Lambda$ denoting the positions of atoms in an idealized solid, or more generally a translation bounded measure $\omega$, its autocorrelation $\gamma$ is defined as the vague limit
\begin{equation}\label{eq-ac}
\gamma= \lim_n \frac{1}{|A_n|} \omega_n*\widetilde{\omega_n}=: \ltebe \omega , \omega \rtebeA \,,
\end{equation}
where $A_n$ is a nice (van Hove) sequence and $\omega_n$ denotes the restriction of $\omega$ to $A_n$. Here, when $\Lambda$ is a pointset, we use $\omega= \delta_{\Lambda}:= \sum_{x \in \Lambda} \delta_x$ in \eqref{eq-ac} to define its autocorrelation. By eventually replacing $A_n$ by a subsequence, the autocorrelation can always be assumed to exist \cite{BL,LSS3}, and is positive definite. Therefore, its Fourier transform $\widehat{\gamma}$ exists and is a positive measure \cite{ARMA1,BF,MoSt}. The measure $\widehat{\gamma}$ is called the diffraction measure of $\Lambda$ and $\omega$, respectively. As any translation bounded measure on the second countable locally compact Abelian group(LCAG) $\widehat{G}$, it has a (unique) refined Lebesgue decomposition
$$
\widehat{\gamma}=(\widehat{\gamma})_{\mathsf{pp}}+\underbrace{(\widehat{\gamma})_{\mathsf{ac}}+(\widehat{\gamma})_{\mathsf{sc}}}_{(\widehat{\gamma})_{\mathsf{c}}}
$$
into a pure point measure $(\widehat{\gamma})_{\mathsf{pp}}$, a measure $(\widehat{\gamma})_{\mathsf{ac}}$ which is absolutely continuous with respect to the Haar measure $\theta_{\widehat{G}}$ and a measure $(\widehat{\gamma})_{\mathsf{sc}}$ which is continuous and singular with respect to the Haar measure $\theta_{\widehat{G}}$.

Systems with pure point spectrum, meaning $(\widehat{\gamma})_{\mathsf{c}}=0$, are now relatively well understood. Pure point diffraction was classified via dynamical spectra \cite{LMS,BL,Gou2,LS2}, via the almost periodicity of the autocorrelation measure \cite{ARMA,BM,Gou,MoSt}. More recently, generalizing previous work in this direction \cite{SOL,SOL2,BM,MS,Gou}, pure point diffraction was shown to be equivalent to mean almost periodicity of the underlying structure \cite{LSS,LSS2}.

\smallskip

The Eberlein decomposition of the autocorrelation measure plays an important role in the study of diffraction of structures. Indeed the autocorrelation $\gamma$ has an unique decomposition $\gamma=\gamma_{\mathsf{s}}+\gamma_0$ into two Fourier transformable measures\cite{ARMA,MoSt} such that
\[
  \widehat{\gamma_{\mathsf{s}}} = (\widehat{\gamma})_{\mathsf{pp}} \,;\,
  \widehat{\gamma_0} = (\widehat{\gamma})_{\mathsf{c}}  \,.
\]
This decomposition allows us study the pure point and continuous diffraction spectrum, respectively, in the real space by studying the two components $\gamma_{\mathsf{s}}$ and $\gamma_0$ of $\gamma$. This approach proved effective in the study of the diffraction spectra of measures with Meyer set support \cite{NS2,NS14,NS11,NS20,NS21} and
the diffraction of compatible random substitutions in 1 dimension \cite{Moll,BSS,Timo}. For Meyer sets, one can further decompose $\gamma_0=\gamma_{0s}+\gamma_{0a}$ in such a way that the Fourier transforms of $\gamma_{0s}$ and $\gamma_{0a}$ are $(\widehat{\gamma})_{\mathsf{sc}}$ and $(\widehat{\gamma})_{\mathsf{ac}}$, respectively \cite{NS20,NS21}.

\smallskip

For many examples of compatible random 1-dimensional Pisot substitutions, one gets a decomposition of the generic element $\omega$ of the hull into two measures $\omega_1$ and $\omega_2$ such that the diffraction of $\omega_1$ and $\omega_2$, respectively, are the pure point and continuous diffraction spectrum, respectively, of $\omega$ \cite{Moll,BSS,Timo}. A similar decomposition hold for 1-dimensional PV substitutions \cite{BS} and for dynamical systems of translation bounded measures \cite{JBA}. It is the one of the goals of this paper to investigate this type of decomposition, at the level of measures and not autocorrelations, in more general settings.

\smallskip

We will approach this question from a different angle. If $\omega=\omega_1+\omega_2$ is such a potential decomposition, and $\cA$ is any van Hove sequence, then there exist a subsequence $\cA'$ of $\cA$ so that (see \eqref{aut lc})
\begin{equation}\label{eq7}
\gamma=\gamma_1+\gamma_2+\ltebe \omega_1 ,\omega_2 \rtebe_{\cA'} +\widetilde{\ltebe \omega_1 ,\omega_2 \rtebe_{\cA'}}  \,,
\end{equation}
where $\gamma, \gamma_1, \gamma_2$ are the autocorrelations of $\omega,\omega_1,\omega_2$,respectively, with respect to $\cA'$ and $\ltebe \omega_1 ,\omega_2 \rtebe_{\cA'}$ denote the twisted Eberlein convolutions (see Def.~\ref{def Ebe conv}). Therefore, we obtain such a decomposition whenever when $\omega_1$ is pure point diffractive, $\omega_2$ has continuous diffraction spectrum and, the following orthogonality with respect to the twisted Eberlein decomposition holds:
\begin{equation}\label{eq1111111}
\ltebe \omega_1 ,\omega_2 \rtebe_{\cA'}  =0 \,.
\end{equation}
Note here also that for positive measures $\omega_1,\omega_2$, \eqref{eq7} implies that $\gamma=\gamma_1+\gamma_2$ is equivalent to \eqref{eq1111111}.

\smallskip
The main goal of the article is to prove that, under suitable conditions, if $\omega_1$ has pure point diffraction spectrum and $\omega_2$ has continuous diffraction spectrum, then \eqref{eq1111111} holds. In particular, the pure point spectrum of $\omega=\omega_1+\omega_2$ is the spectrum of $\omega_1$ and the continuous spectrum of $\omega$ is the spectrum of $\omega_2$. We prove few results in this direction. First, in Thm.~\ref{thm: main} we show that if $\omega_1$ has pure point diffraction spectrum, $\omega_2$ has continuous diffraction spectrum, and the Fourier--Bohr coefficients of $\omega_2$ exist uniformly then \eqref{eq1111111} holds. Next, in Cor.~\ref{cor 2} we show that \eqref{eq1111111} also holds if $\omega_1$ has pure point diffraction spectrum, $\omega_2$ has continuous diffraction spectrum, the Fourier--Bohr coefficients of $\omega_1$ exist uniformly and the Consistent Phase Property (CPP) (see Def.~\ref{def CPP}) holds for $\omega_2$. We continue by restricting our attention to the case when $\omega_1$ is Besicovitch almost periodic, and we show in Cor.~\ref{prop 6.9} that, if $\omega_1$ is Besicovitch almost periodic with respect to $\cA$, $\omega_2$ has continuous diffraction spectrum with respect to $\cA$ and the CPP holds for $\omega_2$ then \eqref{eq1111111} holds. Finally, in Thm.~\ref{thm main 2} we show that if $(\XX,G,m)$ and $(\YY, G, n)$ are ergodic dynamical systems of translation bounded measures, one with pure point spectrum and the other with continuous diffraction spectra, then \eqref{eq1111111} holds for $m$-almost all $\omega_1$ and $n$-almost all $\omega_2$. In particular, some type of generalized Eberlein decomposition holds for such dynamical systems of translation bounded measures such that \eqref{eq1111111} holds almost surely (Thm.~\ref{thm last}).

\smallskip

The Fourier--Bohr coefficients are a central tool we use in many proofs, and along the way we prove various properties for the Fourier--Bohr coefficients of the twisted Eberlein convolution, which are interesting by themselves. First, we show in Thm.~\ref{thm:FB-coefficiceints} that whenever when the twisted Eberlein convolution $\ltebe \mu, \nu \rtebeA$ exists,  $a_\chi(\mu)$ exist uniformly and $a_{\chi}^\cA(\nu)$ exist, then we have the generalized (CPP) relation
\[
a_{\chi}(\ltebe \mu, \nu \rtebeA)= a_{\chi}(\mu) \overline{a_{\chi}(\nu)} \,.
\]
This implies that any measure $\mu$ with uniform Fourier--Bohr coefficients satisfies CPP, a result which was previously proven in $\RR^d$ by \cite{HOF3} and for dynamical systems of translation bounded measures in \cite{Len}. This result has an intriguing consequence: in Thm.~\ref{thm:FB-coefficiceints3} we show that the existence of the twisted Eberlein convolution $\ltebe \mu, \nu \rtebeA$ combined with the uniform existence  $a_\chi(\mu)$ exist uniformly yields the following:
\begin{itemize}
  \item{}  For all $\chi$ with $a_{\chi}(\mu)=0$ we have $a_{\chi}(\ltebe \mu, \nu \rtebeA)=0$.
  \item{}  For all $\chi$ with $a_{\chi}(\mu)\neq 0$, the Fourier--Bohr coefficient $a_\chi^{\cA}(\nu)$ exists and $a_{\chi}(\ltebe \mu, \nu \rtebeA)=a_\chi(\mu)\overline{a_\chi^{\cA}(\nu)}$.
\end{itemize}
It is perhaps surprising that in the second situation we get the existence of  $a_\chi^{\cA}(\nu)$ for free, which hints that there may be an even deeper connection in general between the Fourier--Bohr coefficients and the twisted Eberlein convolution.

\smallskip

The paper is organized as follows. In Section~\ref{sec 2} we briefly review some of the concepts needed in this paper. In Section~\ref{sec 5} we study the connection between the Fourier--Bohr coefficients of $\ltebe \mu, \nu \rtebeA$ and the Fourier--Bohr coefficients of $\mu$ and $\nu$, and discus the consequences of these results to the diffraction is Section~\ref{sect diff}. The orthogonality relation \eqref{eq1111111} is studied in Section~\ref{sec 6} and Section~\ref{sec 7}. We complete the paper by briefly discussion a generalized Eberlein decomposition for measures.

\smallskip

Before starting, let us also emphasize that all results about twisted Eberlein convolution from this paper can automatically be translated to the Eberlein convolution via the relations
\[
\ltebe f , g \rtebeA = f \ebeA \tilde{g}  \,;\, \ltebe \mu , \nu \rtebeA = \mu \ebeA \tilde{\nu}  \,,
\]
Lemma~\ref{fb tilde} and Lemma~\ref{fb:prop}. Since obtaining the equivalent relations for the Eberlein convolution is a trivial exercise, we leave this as an exercise for the reader.

\section{Preliminaries}\label{sec 2}

In this section we briefly review the background information needed in this paper. Throughout the paper $G$ is a second countable locally compact Abelian group (LCAG). We denote by $\widehat{G}$ the dual group of $G$. For a Borel set $A \subseteq G$ we denote the Haar measure of $A$ by $|A|$.

As usual, we denote by $\Cu(G)$ the space of uniformly continuous bounded functions on $G$ and by
$\Cc(G)$ the subspace of $\Cu(G)$ consisting of compactly supported continuous functions. $\cM^\infty(G)$ denotes the space of translation bounded measures.

Since this paper is relying on \cite{LSS3}, instead of collecting here most of the results of this paper, we will assume that the reader is familiar with it and refer to it in the relevant places.
We will also briefly review some of the less common definitions and properties and refer the reader to \cite{TAO,TAO2,LSS,LSS3} for the basic background material.

\subsection{Fourier Bohr coefficients}

In this section we review the basic properties of Fourier--Bohr coefficients. For more details we refer the reader to \cite{LSS}.

We start with the definition.

\begin{definition}  Let $f \in \Cu(G)$ and $\mu \in \cM^\infty(G)$, let $\chi \in \widehat{G}$ and let $\cA$ be van Hove sequence.
We say that the \textbf{Fourier--Bohr coefficient} $a_\chi^{\cA}(f)$ exists if the following limit exists:
\begin{displaymath}
a_\chi^\cA(f):=\lim_n \frac{1}{|A_n|} \int_{A_n} \overline{ \chi(t)}f(t) \dd t \,.
\end{displaymath}
We further say that the Fourier--Bohr coefficient exists \textbf{uniformly} if the limit
\begin{displaymath}
\lim_n \frac{1}{|A_n|} \int_{x+A_n} \overline{ \chi(t)}f(t) \dd t \,,
\end{displaymath}
exists uniformly in $x$.

Same way we say that the \textbf{Fourier--Bohr coefficient} $a_\chi^{\cA}(\mu)$ exists if the following limit exists:
\begin{displaymath}
a_\chi^\cA(\mu):=\lim_n \frac{1}{|A_n|} \int_{A_n} \overline{ \chi(t)} \dd \mu(t) \,.
\end{displaymath}
We further say that the Fourier--Bohr coefficient exists \textbf{uniformly} if the limit
\begin{displaymath}
\lim_n \frac{1}{|A_n|} \int_{x+A_n} \overline{ \chi(t)}\dd \mu(t)\,,
\end{displaymath}
exists uniformly in $x$.
\end{definition}

\medskip

The next result plays a central role in the rest of the paper (compare \cite{EBE,MoSt}).

\begin{lemma}\label{lemma: FB uniform indep van hove2}\cite[Cor.~1.12]{LSS} Let $f \in \Cu(G)$ and $\chi \in G$. If $a_\chi^{\cA}(f)$ exists uniformly, then $a_\chi^{\cB}(f)$ exists uniformly for all van Hove sequences $\cB$ and
\[
a_\chi^{\cB}(f)=a_\chi^{\cA}(f) \,.
\] \qed
\end{lemma}

The same result holds for measures:

\begin{lemma}\cite[Cor.~1.12]{LSS}\label{lemma: FB uniform indep van hove} Let $\mu \in \cM^\infty(G)$ and $\chi \in \widehat{G}$. If the Fourier--Bohr coefficient $a_\chi^\cA(\mu)$ of $\mu \in \cM^\infty(G)$ exists uniformly with respect to the van Hove sequence $\cA=\{ A_n \}$, and $\cB=\{B_n \}$ is another van Hove sequence, then the Fourier--Bohr coefficient of $\mu$ exists uniformly with respect to  $\cB$ and
\[
a_\chi^\cB(\mu)=a_\chi^\cA(\mu)  \,.
\]\qed
\end{lemma}

\medskip

Due to Lemma~\ref{lemma: FB uniform indep van hove2} and Lemma~\ref{lemma: FB uniform indep van hove}, whenever when the Fourier--Bohr coefficient $f \in \Cu(G)$ or $\mu \in \cM^\infty(G)$, respectively, exists uniformly, we can simply write $a_\chi(f)$ and $a_\chi(\mu)$, respectively. We will do this in the remaining of the paper. Moreover, whenever when we write $a_\chi(f)$ and $a_\chi(\mu)$, respectively, it is understood that the Fourier--Bohr coefficients exist uniformly.

\medskip

Since we will often need to refer to the non-zero Fourier--Bohr coefficients, we introduce the following definitions.

\begin{definition}\label{def CPP} We say that
\begin{itemize}
  \item[(a)] $f \in \Cu(G)$ (or $\mu \in \cM^\infty(G)$) has a \textbf{well defined Fourier--Bohr spectrum} with respect to $\cA$ if for all $\chi \in \widehat{G}$ the Fourier--Bohr coefficient $a_{\chi}^\cA(f)$ (or $a^\cA_{\chi}(\mu)$) exists. In this case, the \textbf{Fourier--Bohr spectrum of $f$ (or $\mu$) with respect to $\cA$} is
\begin{align*}
\FB_{\cA}(f)&=\{ \chi \in \widehat{G} : a_\chi^{\cA}(f) \neq 0 \} \,, \\
\FB_{\cA}(\mu)&=\{ \chi \in \widehat{G} : a_\chi^{\cA}(\mu) \neq 0 \} \,.
\end{align*}

If furthermore the Fourier--Bohr coefficients exist uniformly, we say that $f$ or $\mu$, respectively, has \textbf{uniform Fourier--Bohr spectrum}.
  \item[(b)]  We say that $f$ or $\mu$, respectively, has a \textbf{null Fourier--Bohr spectrum with respect to $\cA$} if the Fourier--Bohr spectrum exists with respect to $\cA$ and is empty.
  If the Fourier--Bohr spectrum exists uniformly and is empty we say that $f$ or $\mu$, respectively, \textbf{uniformly-null Fourier--Bohr spectrum}.

\item[(c)] Let  $\mu \in \cM^\infty(G)$ and $\cA$ be a van Hove sequence so that the autocorrelation $\gamma$ of $\mu$ exists with respect to $\cA$. We say that $\mu$ satisfies the \textbf{Consistent Phase Frequency (CPP)} with respect to $\cA$ if the Fourier--Bohr spectrum is well defined with respect to $\cA$ and
\begin{equation}\label{eqn CPP}
\widehat{\gamma}(\{ \chi \})=|a_\chi^{\cA}(\mu)|^2 \quad \forall \chi \in \widehat{G} \qquad \mbox{\bf \rm (CPP) } \,.
\end{equation}
\end{itemize}

\end{definition}

As pointed above, if $f$ or $\mu$ have uniform Fourier--Bohr spectrum, then the Fourier--Bohr spectrum exists for all van Hove sequences, and is independent of the choice of the van Hove sequence. In this case we simply write $\FB(f)$ and $\FB(\mu)$, respectively, for the Fourier--Bohr spectrum.

\begin{remark}
\begin{itemize}
  \item[(a)] Let $\mu \in \cM^\infty(G)$ be a measure whose autocorrelation $\gamma$ exists with respect to $\cA$ and $\mu$ satisfies the CPP.
  Then $\mu$ has null Fourier--Bohr spectrum if and only if $\widehat{\gamma}$ is a continuous measure.
  \item[(b)] We will show later in Thm.~\ref{thm: phase consistency property} that the uniform Fourier--Bohr spectrum implies that CPP holds for each van Hove sequence for which the autocorrelation exist.
\end{itemize}
\end{remark}

\medskip

Next, we study how the Fourier--Bohr coefficients behave with respect to the basic operations on functions.

\begin{lemma}\label{fb tilde}Let $f \in \Cu(G), \chi \in \widehat{G}$ and $\cA$ a van Hove sequence. If $a_\chi^{\cA}(f)$ exists (uniformly) then
$a_\chi^{\cA}(\overline{f}),a_\chi^{-\cA}(\tilde{f}),a_\chi^{-\cA}(f^\dagger)  $ exist (uniformly) and
\begin{align}
a_{\bar{\chi}}^{\cA}(\overline{f}) &=\overline{ a_{\chi}^{\cA}(f)}  \label{R1} \\
a_\chi^{-\cA}(\tilde{f}) &= \overline{a_{\chi}^{\cA}(f)} \nonumber \\
a_{\bar{\chi}}^{-\cA}(f^\dagger)  & =a_{\chi}^{\cA}(f) \,.  \label{R3}
\end{align}
 \end{lemma}
 \begin{proof}
 The computations are straightforward: for each $x \in G$ and $n$ we have
\begin{displaymath}
\frac{1}{|A_n|}\int_{x+A_n} \chi(t) \overline{f(t)} \dd t = \overline{ \frac{1}{|A_n|}\int_{x+A_n} \overline{\chi(t)}f(t) \dd t }  \,.
\end{displaymath}
Taking the limit we get \eqref{R1}.

Next,
\begin{align*}
\frac{1}{|A_n|}\int_{x-A_n} \chi(t)  f^\dagger (t) \dd t & =  \frac{1}{|A_n|}\int_{x-A_n}\overline{\chi(-t)}  f(-t) \dd t
&=  \frac{1}{|A_n|}\int_{-x+A_n}\overline{\chi(s)}  f(s) \dd s
\end{align*}
Taking the limit we get \eqref{R3}.

Finally,
$$
a_\chi^{-\cA}(\tilde{f})=a_\chi^{-\cA}(\overline{f^\dagger})=\overline{a_{\bar{\chi}}^{-\cA}(f^\dagger)}=\overline{a_{\chi}^{\cA}(f)} \,,
$$
with all limits existing (uniformly) by the above.

 \end{proof}

 \medskip

Since, $\widetilde{\mu*\varphi}=\widetilde{\mu}*\tilde{\varphi}$, we can use \cite[Cor.~1.11]{LSS} to transfer these relations to measures.

\begin{coro}\label{fb:prop} Let $ \mu \in \Cu(G), \chi \in \widehat{G}$ and $\cA$ a van Hove sequence. If $a_\chi^{\cA}(\mu)$ exists (uniformly) then
$a_\chi^{\cA}(\overline{\mu}),a_\chi^{-\cA}(\tilde{\mu}),a_\chi^{-\cA}(\mu^\dagger)$ exist (uniformly) and
\begin{align*}
a_{\bar{\chi}}^{\cA}(\overline{\mu}) &=\overline{ a_{\chi}^{\cA}(\mu)}   \\
a_\chi^{-\cA}(\tilde{\mu}) &= \overline{a_{\chi}^{\cA}(\mu)}   \\
a_{\bar{\chi}}^{-\cA}(\mu^\dagger)  & =a_{\chi}^{\cA}(\mu) \,.
\end{align*}
\qed
 \end{coro}

Let us note here in  passing that, when $a_\chi(\mu)$ exists uniformly, we have
\[
a_{\bar{\chi}}(\overline{\mu}) = \overline{ a_{\chi}(\mu)}  \, ; \,
a_\chi(\tilde{\mu}) = \overline{a_{\chi}(\mu)}   \,;\,
a_{\bar{\chi}}(\mu^\dagger)  =a_{\chi}(\mu) \,.
\]

\subsection{Twisted Eberlein convolution of measures}

Here we briefly review the basic properties of Eberlein convolution of measures from \cite{LSS3}.

\begin{definition}\label{def Ebe conv} We say that $f,g : G \to \CC$ have a well defined \textbf{twisted Eberlein convolution} with respect to $\cA$ if for all $t \in G$ the function $s \to f(s)\overline{g(s-t)}$ is locally integrable and the following limit exists
$$
\ltebe f , g \rtebeA (t):= \lim_n \frac{1}{|A_n|} \int_{A_n} f(s) \overline{g(s-t)} \dd s  \,.
$$

We say that $\mu, \nu \in \cM^\infty(G)$ have a well defined \textbf{Eberlein convolution} with respect to $\cA$ if the following vague limit exists
\begin{displaymath}
\ltebe \mu , \nu \rtebeA=\lim_n \frac{1}{|A_n|}\mu|_{A_n}*\widetilde{\nu|_{A_n}} \,.
\end{displaymath}
\end{definition}

Note here that, given $\mu \in \cM^\infty(G)$ and a van Hove sequence $\cA$, the autocorrelation $\gamma$ of $\mu$ exists with respect to $\cA$ if and only if
$\ltebe \mu, \mu \rtebeA$ exists. In this situation we have $\gamma = \ltebe \mu, \mu \rtebeA$.

\begin{remark} One of the primary goals of the paper is to look at the autocorrelation of a measure $\mu$ of the form $\omega=a\mu+b\nu$ for some $\mu, \nu \in \cM^\infty(G)$ and $a,b \in \CC$.

By \cite[Thm.~4.14]{LSS3}, there exists some subsequence $\cA'$ of $\cA$ such that all the convolutions below are well defined and
\begin{align}
  \gamma_\omega &= \ltebe \left(a\mu+b\nu \right) , \left(a\mu+b\nu \right) \rtebeA  \nonumber\\
 &= |a|^2 \ltebe \mu , \mu\rtebe_{\cA'} +|b|^2 \ltebe \mu , \nu\rtebe_{\cA'} +a \bar{b} \ltebe \mu , \nu\rtebe_{\cA'} + b \bar{a} \ltebe \nu , \mu\rtebe_{\cA'}     \nonumber\\
 &= |a|^2\gamma_{\mu} +|b|^2 \gamma_{\nu}+a \bar{b} \ltebe \mu , \nu\rtebe_{\cA'} + b \bar{a} \widetilde{\ltebe \mu , \nu\rtebe_{\cA'}} \label{aut lc} \,,
\end{align}
where $\gamma_{\mu}, \gamma_{\nu}$ are the autocorrelations of $\mu$ and $\nu$, respectively, with respect to $\cA'$. Moreover, \eqref{aut lc} holds for all van Hove subsequences
$\cA'$ of $\cA$ for which all Eberlein convolutions are well defined.

We will show that under certain conditions we have $\ltebe \mu , \nu\rtebe_{\cA'}=0$, leading to
\[
\gamma_{\omega}=  |a|^2\gamma_{\mu} +|b|^2 \gamma_{\nu} \,.
\]
\smallskip
Let us emphasize here that computations of this type are essential to the study of substitution tilings via renormalisation equations (see for example \cite{BaGa,Man,BGrM,BFGR,BGM,BG,BS,BFG} just to name a few), as well as in the study of random substitutions (see for example \cite{Moll,BSS,RS,Timo}).
\end{remark}

\medskip

The following result allows us switch back and forth between the Eberlein convolution of measures and functions.

\begin{lemma}\label{lemma ebe equiv} Let $\mu, \nu \in \cM^\infty(G)$ and $\cA$ a van Hove sequence. Then, the following are equivalent:
\begin{itemize}
\item[(i)] The twisted Eberlein convolution $\ltebe \mu , \nu \rtebeA$ exists.
\item[(ii)] For all $\varphi,\psi \in \Cc(G)$  the following limit exists
\[
M_{\cA}(\left( \mu*\varphi \right)\overline{\left( \nu*\psi \right)}):= \lim_n \frac{1}{|A_n|} \int_{A_n} \left( \mu*\varphi(t) \right)\overline{\left( \nu*\psi(t) \right)} \dd t \,.
\]
\item[(iii)] For all $\varphi,\psi \in \Cc(G)$ the twisted Eberlein convolution $\ltebe \bigl(\varphi*\mu \bigr) ,\left( \psi*\nu \right) \rtebeA$ exists.
\end{itemize}

Moreover, in this case, for all  $\varphi,\psi \in \Cc(G)$ we have
\begin{align*}
M_{\cA}( \mu*\varphi \cdot \overline{ \nu*\psi})&= \bigl( \ltebe \mu , \nu \rtebeA *\varphi *\tilde{\psi}\bigr)(0) =\ltebe \bigl(\varphi*\mu \bigr) \,\left( \psi*\nu \right) \rtebeA(0) \quad \mbox{ and} \\
\ltebe \mu , \nu \rtebeA*\varphi*\tilde{\psi}&= \ltebe \bigl(\varphi*\mu \bigr) ,\left( \psi*\nu \right) \rtebeA \,.
\end{align*}
\end{lemma}
\begin{proof}
(i) $\Leftrightarrow$ (ii) and the equality
\[
M( \mu*\varphi \cdot \overline{ \nu*\psi})=\bigl( \ltebe \mu , \nu \rtebeA *\varphi *\tilde{\psi}\bigr)(0) \,,
\]
follow from \cite[Prop.~1.4]{LSS} .

(i) $\Rightarrow$ (ii) is a standard van Hove computation.

(iii) $\Rightarrow$ (i) and the equality
\[
M_{\cA}( \mu*\varphi \cdot \overline{ \nu*\psi})=\ltebe \bigl(\varphi*\mu \bigr)  ,\left( \psi*\nu \right) \rtebeA(0)
\]
follow immediately from the observation that for all $\varphi, \psi \in \Cc(G)$ and $t \in G$ we have
\begin{align*}
&\frac{1}{|A_n|} \int_{A_n} \bigl(\varphi*\mu \bigr)(s)\widetilde{ \left( \psi*\nu \right)}(t-s) \dd s =\frac{1}{|A_n|} \int_{A_n} \bigl(\varphi*\mu \bigr)(s)\overline{ \left( \psi*\nu \right)}(s-t) \dd s \\
&=\frac{1}{|A_n|} \int_{A_n} \bigl(\varphi*\mu \bigr)(s)\overline{ \tau_t\left( \psi*\nu \right)}(s) \dd s =\frac{1}{|A_n|} \int_{A_n} \bigl(\varphi*\mu \bigr)(s)\overline{  (\tau_t\psi)*\nu }(s) \dd s \,.
\end{align*}

Finally,
\begin{align*}
  \ltebe \bigl(\varphi*\mu \bigr) ,\left( \psi*\nu \right) \rtebeA (t) & =  \ltebe \tau_{-t} \bigl(\varphi*\mu \bigr) ,\left( \psi*\nu \right) \rtebeA (0) \\
  = \ltebe  \bigl(\varphi*(\tau_{-t}\mu) \bigr) ,\left( \psi*\nu \right) \rtebeA (0)  &=\ltebe \tau_{-t} \mu, \nu \rtebeA *\varphi*\tilde{\psi}  (0) \\
  =\left(\tau_{-t}\ltebe  \mu, \nu \rtebeA\right) *\varphi*\tilde{\psi}  (0)&=\ltebe  \mu, \nu \rtebeA *\varphi*\tilde{\psi}(t) \,.
\end{align*}
\end{proof}

\section{Fourier--Bohr coefficients of Eberlein convolution}\label{sec 5}

We start by covering the following result which relates the Fourier--Bohr coefficient of $\ltebe f , g \rtebeA$ to the Fourier--Bohr coefficients of $f,g$. Particular versions of this result can be found in \cite{EBE,ARMA,MoSt,LS,LSS}, and Prop.~\ref{lem:FB coefficiceints} below generalizes all these results.

Before starting let us emphasize that, for $f,g \in \Cu(G)$ for which $\ltebe f, g \rtebeA$ exist, the uniform existence of $a_\chi(\ltebe f,g \rtebeA)$ follows from  \cite[Cor.~5.4]{LSS3} and \cite[Prop.~4.5.9]{MoSt}. In Prop~\ref{lem:FB coefficiceints} below, under the extra assumptions (b,c), we can show directly via a simple standard computation the uniform existence of $a_\chi(\ltebe f,g \rtebeA)$. Since this proof is more elementary, and since \cite[Prop.~4.5.9]{MoSt} relies on very technical results, spanning much of \cite{MoSt}, in Prop.~\ref{lem:FB coefficiceints} we will show directly the uniform existence of the Fourier--Bohr coefficients, without relying on \cite{MoSt}.

\begin{prop}\label{lem:FB coefficiceints} Let $f,g \in \Cu(G), \chi \in \widehat{G}$ and $\cA$ be a van Hove sequence such that
\begin{itemize}
  \item[(a)]The twisted Eberlein convolution $\ltebe f , g \rtebeA$ exists.
  \item[(b)] The Fourier--Bohr coefficient $a_\chi^{\cA}(f)$ exists.
  \item [(c)] The Fourier--Bohr coefficient $a_\chi(g)$ exists uniformly.
\end{itemize}
Then, the Fourier--Bohr coefficient $a_\chi^{\cA}(\ltebe f , g \rtebeA)$ exists uniformly and satisfies
\begin{displaymath}
a_\chi(\ltebe f , g \rtebeA)= a_\chi^{\cA}(f) \overline{a_\chi(g)} \,.
\end{displaymath}
\end{prop}
\begin{proof}

Let $\epsilon >0$. Then, by (b) and (c) there exists an $N_\epsilon$ so that, for all $n,m >N_\epsilon$ and all $x \in G$  we have
\begin{align*}
  \left|\frac{1}{|A_m|} \int_{A_m} \overline{\chi(t)} f(t)- a^{\cA}_{\chi}(f) \right| &< \frac{\eps}{4 \|g \|_\infty+1} \\
  \left|\frac{1}{|A_n|} \int_{x-A_n} \overline{\chi(t)} g(t)- a_{\chi}(g) \right| &< \frac{\eps}{4 \|f \|_\infty+1} \,.
\end{align*}

In particular, for all $m,n >N_\epsilon$ we have for all $x \in G$
\begin{align*}
&\left| \frac{1}{|A_n|}\frac{1}{|A_m|}\int_{A_m}  \int_{s-x-A_n} \chi(u)\overline{g(u)} \overline{\chi(s)} f(s) \dd u\dd s  -  a^{\cA}_{\chi}(f) \overline{a_{\chi}(g)} \right| \\
&\leq \left| \frac{1}{|A_n|}\frac{1}{|A_m|} \int_{A_m}  \int_{s-x-A_n} \chi(u)\overline{g(u)} \overline{\chi(s)} f(s) \dd u\dd s  -  \frac{1}{|A_m|} \int_{A_m} \overline{\chi(s)}f(s) \overline{a^{\cA}_{\chi}(g)} \dd s  \right|\\
&+\left|\frac{1}{|A_m|} \int_{A_m} \overline{\chi(s)}f(s) \overline{a_{\chi}(g)} \dd s  -  a^{\cA}_{\chi}(f)\overline{a_{\chi}(g)} \right| \\
&\leq \left| \frac{1}{|A_m|} \int_{A_m} \left(\frac{1}{|A_n|} \int_{s-x-A_n} \chi(u)\overline{g(u)}  \dd u  - \overline{a_{\chi}(g)} \right)\overline{\chi(s)} f(s) \dd s  \right|\\
&+|a_{\chi}(g)| \left|\frac{1}{|A_m|} \int_{A_m} \overline{\chi(s)}f(s) \dd s   -  a^{\cA}_{\chi}(f) \right| \\
&\leq \frac{1}{|A_m|} \int_{A_m}\left|f(s)\right| \left|\frac{1}{|A_n|} \overline{\int_{s-x-A_n} \overline{\chi(u)} g(u) \dd u-a_{\chi}(g)} \right| \dd s \\
&+|a_{\chi}(g)| \left|\frac{1}{|A_m|} \int_{A_m} \overline{\chi(s)}f(s) \dd s   -  a^{\cA}_{\chi}(f) \right| \\
&\leq \frac{\eps}{4 \|f \|_\infty}\left( \frac{1}{|A_m|} \int_{A_m}\left|  f(u)\right| \dd u \right)  +\frac{\eps}{4 \|g \|_\infty}  |a_{\chi}(g)| \\
&\leq \frac{\eps}{4 \|f \|_\infty+1} \|f \|_\infty +\frac{\eps}{4 \|g \|_\infty+1}  \|g \|_\infty < \frac{\eps}{2} \,.
\end{align*}
Thus, for all $m,n >N_\epsilon$ and $x \in G$ we have
\begin{equation}\label{eq4}
\left| \frac{1}{|A_n|}\frac{1}{|A_m|} \int_{A_m}  \int_{s-x-A_n} \chi(u)\overline{g(u)} \overline{\chi(s)} f(s) \dd u\dd s  -  a^{\cA}_{\chi}(f) \overline{a_{\chi}(g)} \right| < \frac{\epsilon}{2} \,.
\end{equation}
\bigskip
Next, let $n >N_\epsilon$ and $x \in G$ be fixed but arbitrary.

Recall that for all $t \in G$ we have
\[
\ltebe f , g \rtebeA (t)= \lim_m \frac{1}{|A_m|} \int_{A_m}  f(s)  \overline{g(s-t)} \dd s \,.
\]
Define
\[
h_m(t):= \left( \frac{1}{|A_m|} \int_{A_m}  f(s)  \overline{g(s-t)} \dd s  \right) \overline{\chi(t)} 1_{x+A_n}(t) \,.
\]
Then, $h_m(t)$ converges pointwise to $(\ltebe f , g \rtebeA (t)) \overline{\chi(t)} 1_{x+A_n}(t)$ and is bounded by $\| f\|_\infty \| g \|_\infty 1_{x+A_n} \in L^1(G)$. Therefore, by the dominated convergence theorem \cite{Ped,rud2} we have
\[
\lim_m \frac{1}{|A_n|} \int_{x+A_n} h_m(t) \dd t =  \frac{1}{|A_n|} \int_{x+A_n} \overline{\chi(t)}  \ltebe f , g \rtebeA (t) \dd t \,.
\]
Therefore, there exists some $M(n,x, \epsilon)$ such that, for all $m >M(n,x, \epsilon)$ we have
\begin{equation}\label{eq5}
\left|  \frac{1}{|A_n|} \int_{G} h_m(t) \dd t - \frac{1}{|A_n|} \int_{x+A_n} \overline{\chi(t)}  \ltebe f , g \rtebeA(t) \dd t  \right| <\frac{\epsilon}{2} \,.
\end{equation}
Note here that for all $m$ we have
\begin{align*}
\frac{1}{|A_n|} \int_{G} h_m(t) \dd t & =\frac{1}{|A_n|} \int_{x+A_n}\overline{\chi(t)} \frac{1}{|A_m|} \int_{A_m}  f(s) \overline{g(s-t)} \dd s   \dd t \\
&=\frac{1}{|A_n|}\frac{1}{|A_m|}  \int_{x+A_n}\int_{A_m} \overline{\chi(s)} f(s)\chi(s-t)  \overline{g(s-t)}  \dd s   \dd t \\
&\stackrel{\mbox{Fubini}}{=\joinrel=\joinrel=\joinrel=\joinrel=\joinrel=\joinrel= }\frac{1}{|A_n|}\frac{1}{|A_m|} \int_{A_m}  \int_{x+A_n} \overline{\chi(s)} f(s)\chi(s-t)  \overline{g(s-t)}   \dd t \dd s \\
&=\frac{1}{|A_n|}\frac{1}{|A_m|}\int_{A_m}   \overline{\chi(s)} f(s) \left( \overline{\int_{x+A_n} \overline{\chi(s-t)}  g(s-t)    \dd t} \right) \dd s \\
&\stackrel{u=s-t}{=\joinrel=\joinrel=\joinrel=\joinrel= }\frac{1}{|A_n|}\frac{1}{|A_m|}\int_{A_m}   \overline{\chi(s)} f(s) \overline{\left( \int_{s-x-A_n} \overline{\chi(u)}  g(u)    \dd u \right)} \dd s \,.
\end{align*}
Therefore,for all $m > M(n,x, \epsilon)$ we have by \eqref{eq5}
\begin{equation}\label{eq6}
\left| \frac{1}{|A_n|}\frac{1}{|A_m|}\int_{A_m}   \overline{\chi(s)} f(s) \overline{\left( \int_{s-x-A_n} \overline{\chi(u)}  g(u)    \dd u \right)} \dd s  - \frac{1}{|A_n|} \int_{x+A_n} \overline{\chi(t)} \ltebe f, g \rtebeA (t) \dd t  \right| <\frac{\epsilon}{2} \,.
\end{equation}
Now, pick one $m> \max \{ M(n,x, \epsilon),N_\epsilon \}$. By Combining \eqref{eq4} and \eqref{eq6} we get that for this $n,m,x$ we have
\begin{align*}
&\left| \frac{1}{|A_n|} \int_{x+A_n} \overline{\chi(t)}  \ltebe f, g \rtebeA (t) \dd t - a^{\cA}_{\chi}(f) \overline{a_{\chi}(g)}  \right| \\
&\leq  \left| \frac{1}{|A_n|}\frac{1}{|A_m|} \int_{A_m}  \int_{s-x-A_n} \chi(u)\overline{g(u)} \overline{\chi(s)} f(s) \dd u\dd s  -  a^{\cA}_{\chi}(f) \overline{a_{\chi}(g)} \right| \\
 &+\left| \frac{1}{|A_n|}\frac{1}{|A_m|}\int_{A_m}   \overline{\chi(s)} f(s) \overline{ \left( \int_{s-x-A_n} \overline{\chi(u)}  g(u)    \dd u \right)} \dd s  - \frac{1}{|A_n|} \int_{x+A_n} \overline{\chi(t)} \ltebe f , g\rtebeA(t) \dd t  \right| \\
 &\leq \frac{\epsilon}{2}+\frac{\epsilon}{2}=\epsilon \,.
\end{align*}
Therefore, for all $n >N_\epsilon$ and all $x \in G$ we have
\[
\left| \frac{1}{|A_n|} \int_{x+A_n} \overline{\chi(t)}  \ltebe f , g\rtebeA(t) \dd t - a^{\cA}_{\chi}(f) \overline{a_{\chi}(g)}  \right|< \epsilon \,,
\]
which proves the claim.
\end{proof}

Using the $\sim$-symmetry relation  $\ltebe g , f \rtebeA = \widetilde{\ltebe f , g \rtebeA}$, and Lemma~\ref{fb tilde} one immediately gets the symmetric result:
\begin{coro}
Let $f,g \in \Cu(G), \chi \in \widehat{G}$ and $\cA$ be a van Hove sequence such that $\ltebe f , g \rtebeA$ exists, $a_\chi(f)$ exists uniformly and $a^{\cA}_\chi(g)$ exists. Then
\begin{displaymath}
a_\chi(\ltebe f , g \rtebeA)= a_\chi(f) \overline{a^{\cA}_\chi(g)} \,.
\end{displaymath}
\end{coro}

As an interesting, and somewhat surprising consequence we get that, if $a_\chi(g)$ exist uniformly and is non-zero, the existence of $\ltebe f, g \rtebeA$ implies the existence of the
Fourier--Bohr coefficient $a_\chi^{\cA}(f)$.

\begin{theorem}\label{t1} Let $f,g \in \Cu(G), \chi \in \widehat{G}$ and $\cA$ be a van Hove sequence such that the twisted Eberlein convolution
$\ltebe f , g \rtebeA$ exists and the Fourier--Bohr coefficient $a_\chi(g)$ exists uniformly. Then, $a_{\chi}(\ltebe f , g \rtebeA)$ exist uniformly and
\begin{itemize}
  \item[(a)] If $a_\chi(g)=0$ then
\[
a_{\chi}(\ltebe f , g \rtebeA) =0 \,.
\]
  \item[(b)] If $a_\chi(g) \neq 0$ then the Fourier--Bohr coefficient $a_\chi^{\cA}(f)$ exists and
\[
a_{\chi}(\ltebe f , g \rtebeA) =a^{\cA}_{\chi}(f) \overline{a_{\chi}(g)} \,.
\]
\end{itemize}
\end{theorem}
\begin{proof}
The uniform existence of $ a_{\chi}(\ltebe f , g \rtebeA)$ follows from \cite[Cor.~5.4]{LSS3} and \cite[Prop.~4.5.9]{MoSt}.

\textbf{(a)} Consider the sequence
\[
a_n=\frac{1}{|A_n|} \int_{A_n} \overline{\chi(t)} f(t) \dd t \,.
\]
Since $f \in \Cu(G)$, $a_n$ is bounded in $\CC$, and hence has a convergent subsequence $a_{k_n}$.

Let $B_n:= A_{k_n}$ and set $\cB=\{ B_n \}_n$.

Since $\cB$ is a subsequence of $\cA$, and the Fourier--Bohr coefficient $a^{\cB}_{\chi}(f)$ exists, by Proposition~\ref{lem:FB coefficiceints} we get
\[
a_{\chi}( a_{\chi}(\ltebe f , g \rtebeA))=a_{\chi}(\ltebe f , g \rtebeA) = a_{\chi}^\cB(f) \overline{a_\chi(g)}=0 \,.
\]

\medskip
\textbf{(b)} Consider again the sequence
\[
a_n=\frac{1}{|A_n|} \int_{A_n} \overline{\chi(t)} f(t) \dd t \,.
\]
We want to show that $a_n$ converges to $a:=\frac{a_{\chi}(\ltebe f , g \rtebeA)}{a_{\chi}(g)}$.

Since $f \in \Cu(G)$, $a_n$ is bounded in $\CC$, and therefore it is a subset of a compact metric set. Therefore, to prove that $a_n$ converges to $a$ it suffices to show that any convergent subsequence of $a_n$ converges to $a$.

Now, we repeat the argument in (a). Let $a_{k_n}$ be a subsequence of $a_n$ which converges to some $a'$. Let $B_n:= A_{k_n}$ and set $\cB=\{ B_n \}_n$. Then, the Fourier--Bohr coefficient $a_{\chi}^{\cB}(f)$ exists and $ a_{\chi}^{\cB}(f)=a'$. Again by Proposition~\ref{lem:FB coefficiceints} we get
\[
a_{\chi}(\ltebe f , g \rtebeA)=a_{\chi}^\cB(f) a_\chi(g)=a' \cdot a_\chi(g) \,.
\]
Since $a_{\chi}(g) \neq 0$, we get $a'=a$, which proves the claim.
\end{proof}

By using again the $\sim$-symmetry of the twisted Eberlein convolution, we get the following twin result:

\begin{coro} Let $f,g \in \Cu(G), \chi \in \widehat{G}$ and $\cA$ be a van Hove sequence such that
$\ltebe f , g \rtebeA$ exists and the Fourier--Bohr coefficient $a_\chi(f)$ exists uniformly. Then, $a_{\chi}(\ltebe f , g \rtebeA)$ exist uniformly and
\begin{itemize}
  \item[(a)] If $a_\chi(f)=0$ then
\[
a_{\chi}(\ltebe f , g \rtebeA) =0 \,.
\]
  \item[(b)] If $a_\chi(f) \neq 0$ then the Fourier--Bohr coefficient $a_\chi^{\cA}(g)$ exists and
\[
a_{\chi}(\ltebe f , g \rtebeA) =a_{\chi}(f)\overline{a_\chi^{\cA}(g)} \,.
\]
\end{itemize}\qed
\end{coro}

\smallskip

Via the standard trick of convolutions with functions $\varphi, \psi \in \Cc(G)$, the above results extend trivially to measures. We start by
translating Prop.~\ref{lem:FB coefficiceints} to measures:

\begin{theorem}\label{thm:FB-coefficiceints} Let $\mu, \nu \in \cM^\infty(G), \chi \in \widehat{G}$ and $\cA$ be a van Hove sequence such that
\begin{itemize}
  \item[(a)]The twisted Eberlein convolution $\ltebe \mu , \nu \rtebeA$ exists.
  \item[(b)] The Fourier--Bohr coefficient $a_\chi^{\cA}(\mu)$ exists.
  \item [(c)] The Fourier--Bohr coefficient $a_\chi(\nu)$ exists uniformly.
\end{itemize}
Then, the Fourier--Bohr coefficient $a_\chi^{\cA}(\ltebe \mu , \nu \rtebeA)$ exists uniformly and satisfies
\begin{displaymath}
a_\chi(\ltebe \mu , \nu \rtebeA)= a_\chi^{\cA}(\mu) \overline{a_\chi(\nu)} \,.
\end{displaymath}
\end{theorem}
\begin{proof}
Pick some $\varphi \in \Cc(G)$ such that $\widehat{\varphi}(\chi)\neq0$, which exists by \cite{BF,MoSt}.

By Lemma~\ref{lemma ebe equiv}, $\ltebe \bigl(\varphi*\mu \bigr) ,\left( \varphi*\nu \right) \rtebeA$ exists and
\[
\ltebe \bigl(\varphi*\mu \bigr) ,\left( \varphi*\nu \right) \rtebeA =\ltebe \mu , \nu \rtebeA*\varphi*\tilde{\varphi} \,.
\]
The claim follows now from \cite[Cor.~1.1]{LSS} and Prop~\ref{lem:FB coefficiceints}.

\end{proof}

Once again, via $\sim$-symmetry we get:

\begin{coro}
Let $\mu,\nu \in \Cu(G), \chi \in \widehat{G}$ and $\cA$ be a van Hove sequence such that $\ltebe \mu , \nu \rtebeA$ exists, $a_\chi(\mu)$ exists uniformly and $a^{\cA}_\chi(\nu)$ exists. Then
\begin{displaymath}
a_\chi(\ltebe \mu , \nu \rtebeA)= a_\chi(\mu) \overline{a^{\cA}_\chi(\nu)} \,.
\end{displaymath}
\end{coro}

Next, we extend Thm.~\ref{t1}.

\begin{theorem}\label{thm:FB-coefficiceints3} Let $\mu, \nu \in \cM^\infty(G), \chi \in \widehat{G}$ and $\cA$ be a van Hove sequence such that $\ltebe \mu , \nu \rtebeA$ exists and the Fourier--Bohr coefficient $a_\chi(\nu)$ exists uniformly.
Then, $a_{\chi}(\ltebe \mu , \nu \rtebeA)$ exist uniformly and
\begin{itemize}
  \item[(a)] If $a_\chi(\nu)=0$ then
\[
a_{\chi}(\ltebe \mu , \nu \rtebeA) =0 \,.
\]
  \item[(b)] If $a_\chi(\nu) \neq 0$ then the Fourier--Bohr coefficient $a_\chi^{\cA}(\mu)$ exists and
\[
a_{\chi}(\ltebe \mu , \nu \rtebeA) =a^{\cA}_{\chi}(\mu) \overline{a_{\chi}(\nu)} \,.
\]
\end{itemize}\qed
\end{theorem}

\begin{coro} Let $\mu, \nu \in \cM^\infty(G), \chi \in \widehat{G}$ and $\cA$ be a van Hove sequence such that $\ltebe \mu , \nu \rtebeA$ exists and the Fourier--Bohr coefficient $a_\chi(\mu)$ exists uniformly.
Then, $a_{\chi}(\ltebe \mu , \nu \rtebeA)$ exist uniformly and
\begin{itemize}
  \item[(a)] If $a_\chi(\mu)=0$ then
\[
a_{\chi}(\ltebe \mu , \nu \rtebeA) =0 \,.
\]
  \item[(b)] If $a_\chi(\mu) \neq 0$ then the Fourier--Bohr coefficient $a_\chi^{\cA}(\nu)$ exists and
\[
a_{\chi}(\ltebe \mu , \nu \rtebeA) =a^{\cA}_{\chi}(\mu) \overline{a_{\chi}(\nu)} \,.
\]
\end{itemize}\qed
\end{coro}

\section{The diffraction measure}\label{sect diff}

Here we list some consequences of the Fourier--Bohr results to the diffraction. We start by showing that the uniform existence of the Fourier--Bohr coefficient implies that the intensity of the Bragg peak at that point is the square of the absolute value of the Fourier--Bohr coefficient. For $G=\RR^d$ this result is well known, see \cite{Hof2}. For uniquely ergodic dynamical systems the result also follows from \cite[Thm.~5]{Len}.

\begin{theorem}\label{thm: phase consistency property} Let $\mu \in \cM^\infty(G)$ let $\cA$ be a van Hove sequence, and let $\gamma$ be the autocorrelation of $\mu$ with respect to $\cA$. If for some $\chi \in \widehat{G}$ the Fourier--Bohr coefficient $a_\chi(\mu)$ exists uniformly, then we get
\begin{displaymath}
\widehat{\gamma}(\{ \chi \}) = \left| a_\chi(\mu) \right|^2 \,.
\end{displaymath}

In particular, if $\mu$ has uniform Fourier--Bohr spectrum, then $\mu$ satisfies CPP \eqref{eqn CPP}.
\qed
\end{theorem}

\medskip

It follows from Thm.~\ref{thm: phase consistency property} that, for measures with uniform Fourier--Bohr spectrum, null Fourier--Bohr spectrum is equivalent to continuous diffraction spectrum.

\begin{coro}\label{cor2} Let $\mu \in \cM^\infty(G)$ be a measure with uniform Fourier--Bohr spectrum. Then, the following are equivalent.
\begin{itemize}
  \item[(i)] The measure $\mu$ has null Fourier--Bohr spectrum with respect one van Hove sequence $\cA$.
  \item[(ii)] The measure $\mu$ has null Fourier--Bohr spectrum with respect all van Hove sequences.
  \item[(iii)] There exists an autocorrelation $\gamma$ of $\mu$ such that $\widehat{\gamma}_{\mathsf{pp}}=0$.
  \item[(iv)] For every autocorrelation $\gamma$ of $\mu$ we have $\widehat{\gamma}_{\mathsf{pp}}=0$.
\end{itemize}
\qed
\end{coro}

\smallskip

Let us also note in passing that if the Fourier--Bohr spectrum exists uniformly, then the Bragg diffraction spectrum is independent of the choice of the van Hove sequence.

\begin{coro}Let $\mu \in \cM^\infty(G)$ be a measure with uniform Fourier--Bohr spectrum. Let $\gamma_1, \gamma_2$ be autocorrelations of $\mu$ with respect to the van Hove sequences $\cA$ and $\cB$, respectively. Then
\[
(\widehat{\gamma_1})_{\mathsf{pp}} =(\widehat{\gamma_2})_{\mathsf{pp}}   \,.
\]
Moreover, if $\gamma$ is the autocorrelation of $\mu$ with respect to some van Hove sequence $\cA$, the Bragg diffraction spectrum $\mathbb{B}_{\cA}(\mu)$ does not depend on $\cA$ and satisfies
\[
\FB(\mu)=\mathbb{B}_{\cA}(\mu):=\{ \chi \in \widehat{G}: \widehat{\gamma}(\{\chi \}) \neq 0 \} \,.
\]
\qed
\end{coro}

The following example shows that in this situation, the continuous diffraction spectrum can depend on the van Hove sequence, even if the Fourier--Bohr spectrum exists uniformly.

\begin{example} Consider the Thue-Morse measure $\omega$ \cite[Sect.4.6]{TAO}. Define
$$
\mu= \delta_{\ZZ} + \omega|_{[0, \infty)} \,,
$$
that is $\mu= \sum_{n \in \ZZ} c(n) \delta_n$ where
$$
c(n)=
\left\{
\begin{array}{c c}
  1 &\mbox{ if } n \in \ZZ \mbox{ and } n< 0 \\
  1+\omega(\{n\}) &\mbox{ if } n \in \ZZ \mbox{ and } n \geq  0\\
\end{array}
\right.
$$
Now, the Fourier--Bohr coefficients of $\omega$ exist uniformly. Moreover, since $\omega$ has (singular) continuous diffraction spectrum \cite[Thm.~10.1]{TAO}, all the Fourier--Bohr coefficients of $\omega$ are zero.

It follows immediately that the Fourier--Bohr coefficients $a_\chi(\mu)$ exist uniformly and
$$
a_\chi(\mu)=
\left\{
\begin{array}{c c}
  1 &\mbox{ if } \chi \in \ZZ\\
  0 &\mbox{ otherwise} \,.
\end{array}
\right.
$$
In particular, if the autocorrelation $\gamma_{\cA}$ of $\mu$ exists with respect to some van Hove sequence $\cA$ then
$$
(\reallywidehat{\gamma_{\cA}})_{\mathsf{pp}}= \delta_{\ZZ} \,.
$$
Now, consider the van Hove sequences $B_{n}=[-n,n^2]$ and $C_n=[-n^2,n]$. Then, an easy computation shows that the autocorrelations $\gamma_{\cB}$ and $\gamma_{\cC}$ of $\mu$ exist with respect to these sequences and
\begin{align*}
  \gamma_{\cB} &=\delta_{\ZZ}+\gamma_{\rm TM} \\
  \gamma_{\cC}&=\delta_{\ZZ} \,,
\end{align*}
where $\gamma_{\rm TM}$ is the measure from \cite[Thm.~10.1]{TAO}.

In particular, we have
\begin{align*}
 (\reallywidehat{ \gamma_{\cB}})_{\mathsf{c}} &=\widehat{\gamma_{\rm TM}} \neq 0 \\
  (\reallywidehat{ \gamma_{\cC}})_{\mathsf{c}}&= 0  \,.
\end{align*}

\end{example}

\section{Orthogonality with respect to twisted Eberlein convolution}\label{sec 6}

In this section we establish the first orthogonality with respect to the twisted Eberlein convolution. The key for this result will be Prop.~\ref{prop 1} and Prop.~\ref{prop 3} below.

\begin{prop}\label{prop 1} Let $\mu, \nu \in \cM^\infty$ and $\cA$ a van Hove sequence so that
\begin{itemize}
  \item{} $\ltebe \mu ,\nu \rtebeA$ exists.
  \item{} The Fourier--Bohr coefficients $a_{\chi}(\mu)$ exist uniformly for all $\chi \in \widehat{G}$ and satisfy
\[
  a_\chi(\mu)=0 \quad \forall \chi \in \widehat{G} \,.
\]
\end{itemize}
Then, $\reallywidehat{\ltebe \mu ,\nu \rtebeA}$ is a continuous measure.
\end{prop}

\begin{proof}
By Thm.~\ref{thm:FB-coefficiceints3}, $a_{\chi}(\ltebe \mu ,\nu \rtebeA)$ exist and
\[
a_{\chi}(\ltebe \mu ,\nu \rtebeA) =0 \qquad \forall \chi \in \widehat{G}\,.
\]
Then, $\ltebe \mu, \nu \rtebeA$ is Fourier transformable by \cite[Cor.~5.2]{LSS3}. Therefore, by \cite[Thm.~4.10.14]{MoSt} we have
\[
\reallywidehat{ \ltebe \mu ,\nu \rtebeA}(\{ \chi \})=a_{\chi}(\ltebe \mu ,\nu \rtebeA) =0 \qquad \forall \chi \in \widehat{G} \,.
\]
\end{proof}

As an immediate consequence we get the following result.

\begin{coro}Let $\mu \in \cM^\infty(G)$ and $\cA$ a van Hove sequence so that
\begin{itemize}
  \item{} The autocorrelation $\gamma$ of $\mu$ exists and $\widehat{\gamma}$ is continuous.
  \item{} The Fourier--Bohr coefficients $a_{\chi}(\mu)$ exist uniformly for all $\chi \in \widehat{G}$.
\end{itemize}
Let $\nu \in \cM^\infty(G)$ be any measure and $\cA'$ be any subsequence of $\cA$ such that $\ltebe \mu, \nu\rtebe_{\cA'}$  exists.
Then  $\reallywidehat{\ltebe \mu ,\nu \rtebe_{\cA'}}$ is a continuous measure.
\end{coro}
\begin{proof}
By Cor.~\ref{cor2} $\mu$ has null Fourier--Bohr spectrum. The claim follows from Prop.~\ref{prop 1}.
\end{proof}

\smallskip

Using the $\sim$-symmetry and Cor.~\ref{fb:prop} we also get the following twin result.

\begin{prop}\label{prop 3}Let $\mu, \nu \in \cM^\infty$ and $\cA$ a van Hove sequence so that
\begin{itemize}
  \item{} $\ltebe \mu ,\nu \rtebeA$ exists.
  \item{} The Fourier--Bohr coefficients $a_{\chi}(\nu)$ exist uniformly for all $\chi \in \widehat{G}$ and satisfy
\[
  a_\chi(\nu)=0 \quad \forall \chi \in \widehat{G} \,.
\]
\end{itemize}
Then, $\reallywidehat{\ltebe \mu ,\nu \rtebeA}$ is a continuous measure. \qed
\end{prop}

The equivalent results for functions are true and can be proven exactly  the same way. We list them here.

\begin{lemma} Let $f,g \in \Cu(G)$ and $\cA$ be so that $\ltebe f, g \rtebeA$ exists.
\begin{itemize}
  \item[(a)] If $f$ has uniform null Fourier--Bohr spectrum, then $\ltebe f, g \rtebeA$ has uniform null Fourier--Bohr spectrum.
  \item[(b)] If $f$ has uniform null Fourier--Bohr spectrum, then $\ltebe f, g \rtebeA$ has uniform null Fourier--Bohr spectrum.
\end{itemize}
\end{lemma} \qed

\medskip

We can now prove the following result, which is one of the main results in the paper.

\begin{theorem}\label{thm: main} Let $\mu, \nu \in \cM^\infty(G)$ and $\cA$ a van Hove sequence with the following properties:
\begin{itemize}
  \item {} The autocorrelation $\gamma_{\mu}$ of $\mu$ exists with respect to $\cA$ and $\reallywidehat{\gamma_{\mu}}$ is pure point.
  \item {} The autocorrelation $\gamma_{\nu}$ of $\nu$ exists with respect to $\cA$ and $\reallywidehat{\gamma_{\nu}}$ is continuous.
  \item {} The Fourier--Bohr coefficients $a_\chi(\nu)$ exist uniformly for all $\chi \in \widehat{G}$.
\end{itemize}
Then,
\[
\ltebe \mu, \nu \rtebeA =\ltebe \nu, \mu \rtebeA=0 \,.
\]
In particular, for all $a,b \in \CC$ the autocorrelation $\gamma_{\omega}$ of $\omega=a\mu+b\nu$ exists with respect to $\cA$ and
\begin{align*}
  \left(\reallywidehat{\gamma_{\omega}}\right)_{\mathsf{pp}} & = |a|^2 \reallywidehat{\gamma_{\mu}} \\
  \left(\reallywidehat{\gamma_{\omega}}\right)_{\mathsf{c}} & = |b|^2 \reallywidehat{\gamma_{\nu}} \,.
\end{align*}
\end{theorem}
\begin{proof}
By \cite[Cor.~4.7 and Lemma~4.13]{LSS3} to show that $\ltebe \mu, \nu \rtebeA =0$, it suffices to show that for all subsequences $\cB$ of $\cA$ for which  $\ltebe \mu, \nu \rtebe_{\cB}$ exists, it is zero.

Let $\cB$ be any subsequence of $\cA$ such that $\ltebe \mu, \nu \rtebe_{\cB}$ exists. Since the Fourier--Bohr coefficients $a_\chi(\nu)$ exist uniformly and $\reallywidehat{\gamma_{\nu}}$ is continuous, $\nu$ has null Fourier--Bohr spectrum by Cor.~\ref{cor2} .  Therefore, by  Prop.~\ref{prop 1} we have $\ltebe \mu, \nu \rtebe_{\cB} \in \WAP_0(G)$.

Moreover, by \cite[Thm.~2.13]{LSS}, since $\reallywidehat{\gamma_{\mu}}$ is pure point, we have $\mu \in \MAP_{\cA}(G)$, and hence $\ltebe \mu, \nu \rtebe_{\cB} \in \SAP(G)$ by\cite[Thm~4.20]{LSS3}. Therefore, by \cite[Thm.~4.10.10]{MoSt} we have
\[
\ltebe \mu, \nu \rtebe_{\cB} \in \SAP(G) \cap \WAP_0(G) = \{ 0\} \,.
\]
The last claim follows from \eqref{aut lc}.
\end{proof}

\medskip

Similarly we get the following result. Note here that since the (CPP) may not hold for $\nu$, we cannot relate the null Fourier--Bohr spectrum of $\nu$ with its diffraction spectra.

Since the proof is almost identical to the one of Thm.~\ref{thm: main}, we skip it.

\begin{prop}\label{prop 2} Let $\mu, \nu \in \cM^\infty(G)$ and $\cA$ a van Hove sequence with the following properties:
\begin{itemize}
  \item {} The autocorrelation $\gamma_{\mu}$ of $\mu$ exists with respect to $\cA$ and $\reallywidehat{\gamma_{\mu}}$ is pure point.
  \item {} The Fourier--Bohr coefficients $a_\chi(\mu)$ exist uniformly for all $\chi \in \widehat{G}$.
  \item {} For all $\chi \in \FB(\mu)$ the Fourier--Bohr coefficient $a_\chi^{\cA}(\nu)$ exists and $a_\chi^{\cA}(\nu)=0$.
\end{itemize}
Then, the following twisted Eberlein convolutions exist and
\[
\ltebe \mu, \nu \rtebeA=\ltebe \nu, \mu \rtebeA=0 \,.
\]\qed
\end{prop}

Note in here that Prop.~\ref{prop 2} only requires the existence of the Fourier--Bohr coefficients of $\nu$ only for $\chi \in \FB(\mu)$.
In particular, if the Fourier--Bohr spectrum of $\nu$  exists, we get:

\begin{coro} Let $\mu, \nu \in \cM^\infty(G)$ and $\cA$ a van Hove sequence with the following properties:
\begin{itemize}
  \item {} The autocorrelation $\gamma_{\mu}$ of $\mu$ exists with respect to $\cA$ and $\reallywidehat{\gamma_{\mu}}$ is pure point.
  \item {} The measure $\mu$ has uniform Fourier--Bohr spectrum.
  \item {}  $\nu$ has null Fourier--Bohr spectrum with respect to $\cA$.
\end{itemize}
Then, the following twisted Eberlein convolutions exist and
\[
\ltebe \mu, \nu \rtebeA=\ltebe \nu, \mu \rtebeA=0 \,.
\]\qed
\end{coro}

Let $\nu$ be a measure whose autocorrelation and Fourier--Bohr spectrum exist with respect to $\cA$ and such that CPP hold. Then $\nu$ has null Fourier--Bohr spectrum if and only if
the diffraction spectrum is continuous. Therefore, we have:

\begin{coro}\label{cor 2} Let $\mu, \nu \in \cM^\infty(G)$ and $\cA$ a van Hove sequence with the following properties:
\begin{itemize}
  \item {} The autocorrelation $\gamma_{\mu}$ of $\mu$ exists with respect to $\cA$ and $\widehat{\gamma_{\mu}}$ is pure point.
  \item {} The autocorrelation $\gamma_{\nu}$ of $\nu$ exists with respect to $\cA$ and $\widehat{\gamma_{\nu}}$ is continuous.
  \item {} The Fourier--Bohr coefficients $a_\chi(\mu)$ exist uniformly for all $\chi \in \widehat{G}$.
  \item {}  $\nu$ satisfies the CPP.
\end{itemize}
Then, the following Eberlein convolutions exist and
\[
\ltebe \mu, \nu \rtebeA=\ltebe \nu, \mu \rtebeA=0 \,.
\]\qed
\end{coro}

\medskip

Let us complete the section by showing that any measure $\mu$ which satisfies the conditions of Prop.~\ref{prop 2} is Besicovitch almost periodic.

\begin{lemma}\label{lem3} Let $\mu$ be a measure so that  autocorrelation $\gamma_{\mu}$ of $\mu$ exists with respect to $\cA$ and $\reallywidehat{\gamma_{\mu}}$ is pure point.
  If the Fourier--Bohr coefficients $a_\chi(\mu)$ exist uniformly for all $\chi \in \widehat{G}$, then $\mu \in \Bap_{\cA}(G)$.
\end{lemma}
\begin{proof}
By Thm.~\ref{thm: phase consistency property} the (CPP) holds. Then, $\mu \in \Bap_{\cA}(G)$ by \cite[Thm.3.36]{LSS}.
\end{proof}


\section{Orthogonality to Besicovitch almost periodic measures}\label{sec 7}

In this section, by using some recent results from \cite{LSS}, we can give some results similar to the ones in Section~\ref{sec 6}. Our proofs rely on the following result.

\begin{theorem}\cite[Prop.~3.26]{LSS3} Let $\cA$ be a van Hove sequence.
\begin{itemize}
  \item[(a)] Let $f,g \in \Cu(G)$ be so that $g \in \bap_{\cA}(G)$. If $a_\chi^\cA(f)$ exist for all $\chi$ and $\ltebe f,g \rtebeA$ exist then $\ltebe f,g \rtebeA \in SAP(G)$ and
  \[
  a_\chi^{\cA}(\ltebe f,g \rtebeA)= a_\chi^\cA(f) \overline{ a_\chi^\cA(g)} \qquad \forall \chi \in \widehat{G} \,.
  \]
  \item[(b)]Let $\mu, \nu \in \cM^\infty(G)$ be so that $\mu \in \Bap_{\cA}(G)$. If $a_\chi^\cA(\mu)$ exist for all $\chi$ and $\ltebe \mu,\nu \rtebeA$ exist then $\ltebe \mu , \nu \rtebeA \in \SAP(G)$ and
  \[
  a_\chi^{\cA}(\ltebe \mu, \nu  \rtebeA)= a_\chi^\cA(\mu) \overline{ a_\chi^\cA(\nu)} \qquad \forall \chi \in \widehat{G} \,.
  \]
\end{itemize}\qed
\end{theorem}

As consequences we get (compare \cite[Thm.~4.5]{BS}).

\begin{coro}\label{prop 6.9} Let $\cA$ be a van Hove sequence in $G$ and $\mu \in \Bap_{\cA}(G)\cap \cM^\infty(G)$ and let $\nu$ be any measure such that for all $\chi \in \FB_{\cA}(\mu)$ the Fourier Bohr coefficient $a_{\chi}^\cA(\nu)$ exists and $a_{\chi}^\cA(\nu)=0$. Then, the following twisted Eberlein convolutions exist and
\[
\ltebe \mu , \nu \rtebeA =\ltebe \nu , \mu \rtebeA =0 \,.
\]
In particular, \eqref{eq1111111} holds.
\end{coro}

Note that Cor.~\ref{prop 6.9} together with Lemma~\ref{lem3} can be used to give an alternate proof to Cor~\ref{cor 2}.

\section{The generalized Eberlein decomposition}

Here we discuss the existence of the so called generalized Eberlein decomposition, the known progress in this direction as well as the importance of any such result for diffraction theory.
We start with the following definition.

\begin{definition} Let $\omega \in \cM^\infty(G)$ and $\cA$ a van Hove sequence. We say that $\omega$ admits a \textbf{generalized Eberlein decomposition with respect to $\cA$} if there exist some measures
$\omega_{\mathsf{s}} , \omega_0 \in \cM^\infty(G)$ such that
\[
\omega=\omega_{\mathsf{s}}+\omega_0 \,.
\]
\begin{itemize}
  \item[(a)] We have $\omega_{\mathsf{s}} \in \Bap_{\cA}(G)$.
  \item[(b)] The twisted Eberlein convolution $\ltebe \omega_{\mathsf{s}}, \omega_0 \rtebeA$ exists and $\ltebe \omega_{\mathsf{s}}, \omega_0 \rtebeA=0$.
  \item[(c)] The autocorrelation $\Psi$ of $\omega_0$ exists with respect to $\cA$ and $\widehat{\Psi}$ is purely continuous measure.
\end{itemize}

If the measures $\omega_{\mathsf{s}}, \omega_0$ can be further chosen to have uniform Fourier--Bohr spectrum, then we say that $\omega$ admits a \textbf{uniform generalized Eberlein decomposition with respect to $\cA$}.
\end{definition}

\begin{remark} 
\begin{itemize}
  \item[(a)] If $\omega$ admits a uniform generalized Eberlein decomposition with respect to $\cA$, then the Fourier--Bohr coefficients $a_{\chi}^\cA(\omega), a_{\chi}^\cA(\omega_{\mathsf{s}}),a_{\chi}^\cA(\omega_0)$ exist uniformly. In particular, $\omega, \omega_{\mathsf{s}}, \omega_0$ satisfy the (CPP) and, for all $\chi \in \widehat{G}$ we have
\[
  a_{\chi}(\omega) =a_\chi(\omega_{\mathsf{s}}) \quad ; \quad   a_{\chi}(\omega_0) =0 \,.
\]
  \item[(b)] If $\omega$ admits a uniform generalized Eberlein decomposition with respect to $\cA$, one can further ask if there exists a generalized Eberlein decomposition 
  such that $\omega_{\mathsf{s}} \in \Wap(G)$. By \cite[Thm.~4.15]{LSS}, this is equivalent to the decomposition $\omega=\omega_{\mathsf{s}}+\omega_0$ being a generalized Eberlein decomposition with respect to all van Hove sequences.
\end{itemize}

\end{remark}

Let us start by looking at a situation where the uniform generalized Eberlein decomposition exists (see \cite{BS} for definitions and notations):

\begin{theorem}\cite[Thm.~5.1]{BS}  Let $\Lambda= \dot\bigcup_{1\leqslant i \leqslant N}\, \Lambda_i$ be a
  typed point set generated from of a primitive PV inflation rule in
  one dimension that is aperiodic and has a PV unit as inflation
  factor, and consider the corresponding natural CPS\/
  $(\RR, \RR^m , \cL)$ that emerges via the classic Minkowski
  embedding of the module spanned by
  the points.  Let\/ $W_i$ be the attractors of the induced,
  contractive iterated function system for the windows in internal
  space, and set
\[
  \alpha_i := \frac{\dens(\Lambda^{}_{i})}{\dens (\cL)
    \vol(W^{}_{ i})} \, , \quad
  (\delta_{\Lambda_i})_{\mathsf{s}} \, :=\, \alpha_i \, \delta_{\oplam(W_i)}
    \, , \quad \text{and} \quad
    (\delta_{\Lambda_i})_0 \, := \, \delta^{}_{\Lambda_i}   -  (\delta_{\Lambda_i})_{\mathsf{s}} \,.
\]
Then,
\[
\delta^{}_{\Lambda_i}= (\delta_{\Lambda_i})_{\mathsf{s}}+(\delta_{\Lambda_i})_0
\]
is an uniform generalized Eberlein decomposition with respect to any symmetric van Hove sequence. Moreover, for all $1 \leq i,j \leq N$ we have
\[
\ltebe \left(\delta_{\Lambda_i}\right)_{s} , \left(\delta_{\Lambda_j}\right)_0 \rtebeA =0 \,.
\]\qed
\end{theorem}

Let us next state various consequences of the existence of the (uniform) generalized Eberlein decomposition.

\begin{lemma} Let $\omega \in \cM^\infty(G)$ and $\cA$ a van Hove sequence be so that $\omega$ admits a generalized Eberlein decomposition with respect to $\cA$. Then,
\begin{itemize}
  \item[(a)] The autocorrelations $\eta, \Psi$ of $\omega_{\mathsf{s}}, \omega_0$ exist with respect to $\cA$.
  \item[(b)] $  \eta=\gamma_{\mathsf{s}} \,;\, \Psi = \gamma_0 \,.$
  \item[(c)]$  \widehat{\eta}=\left(\widehat{\gamma}\right)_{\mathsf{pp}} \,;\,   \widehat{\Psi}=\left(\widehat{\gamma}\right)_{\mathsf{c}}$. 
\end{itemize}
\end{lemma}
\begin{proof}
\textbf{(a)}
Since $\omega_{\mathsf{s}} \in \Bap_{\cA}(G)$, its autocorrelation $\eta$ exists with respect to $\cA$ and $\eta \in \SAP(G)$ \cite[Thm.~3.36]{LSS}.

\textbf{(b),(c)} follow from \cite[Thm.~4.10.10 and Thm.~4.10.12]{MoSt}.

\end{proof}

Next, we give an equivalent definition for the uniform generalized Eberlein decomposition, which is an immediate consequence of Thm.~\ref{thm: phase consistency property}.

\begin{prop}\label{prop4} Let $\omega, \mu, \nu \in \cM^\infty(G)$ be so that the Fourier--Bohr coefficients of $\omega$ exist uniformly. Then
\[
\omega_{\mathsf{s}}=\mu \,;\, \omega_0= \nu
\]
is a uniform generalized Eberlein decomposition with respect to $\cA$ if and only if
\begin{itemize}
  \item[(a)] $\mu \in \Bap_{\cA}(G)$.
  \item[(b)] The autocorrelation $\Psi$ of $\omega_0$ exists with respect to $\cA$.
  \item[(c)] The Fourier--Bohr spectrum of $\nu$ is uniformly null.
\end{itemize}\qed
\end{prop}

\begin{remark} Prop.~\ref{prop4} suggests that one may also ask for the existence of a "weak" generalized Eberlein decomposition
\[
\omega= \omega_{\mathsf{s}} +\omega_0
\]
such that
\begin{itemize}
  \item{} $\omega_{\mathsf{s}} \in \Bap_{\cA}(G)$.
  \item{} $\omega_0$ has null Fourier--Bohr spectrum with respect to $\cA$.
\end{itemize}
Note that this implies the orthogonality condition
\[
\ltebe \omega_{s}, \omega_0 \rtebeA =0
\]
together with
\[
\gamma_{\omega}=\gamma_{\omega_{\mathsf{s}}}+\gamma_{\omega_0} \,;\, \reallywidehat{\gamma_{\omega}}=\reallywidehat{\gamma_{\omega_{\mathsf{s}}}}+\reallywidehat{\gamma_{\omega_0}}\,.
\]
While $\reallywidehat{\gamma_{\omega_{\mathsf{s}}}}$ is a pure point measure, there is no guarantee that $\reallywidehat{\gamma_{\omega_0}}$ is continuous, making this type of decomposition
less interesting for the study of diffraction spectra. In fact, if $\reallywidehat{\gamma_{\omega_0}}$ is a continuous measure, then CPP holds and the above decomposition is a generalized Eberlein decomposition, which emphasizes that there is no need
to introduce this new concept if one is only interested in the diffraction spectra. For this reason, we do not look at this situation here.
\end{remark}

\subsection{Dynamical systems of translation bounded measures}

Let us start by stating the following consequence of our results to dynamical systems of translation bounded measures (see \cite{BL} for definition).

\begin{theorem}\label{thm main 2} Let $(\XX, G, m)$ and $(\YY, G, n)$ be two ergodic dynamical systems of translation bounded measures, and let $\gamma$ and $\eta$ be the autocorrelations of $(\XX,m)$ and $(\YY,n)$, respectively. Let $\cA$ be any van Hove sequence along which the ergodic theorem holds. If $\widehat{\gamma}$ is a pure point measure and  $\widehat{\eta}$ is a continuous measure, then, there exists sets $X \subseteq \XX, Y \subseteq \YY$ with the following properties:
\begin{itemize}
  \item[(a)] The sets $X,Y$ satisfy $m(X)=n(Y)=1$.
  \item[(b)] For all $\mu \in X$ the autocorrelation of $\mu$ with respect to $\cA$ is $\gamma$.
  \item[(c)] For all $\nu \in Y$ the autocorrelation of $\nu$ with respect to $\cA$ is $\eta$.
  \item[(d)] For all $\chi \in \mathbb{B}:= \{ \psi \in \widehat{G} : \widehat{\gamma}(\{ \chi \}) \neq 0 \}$ and all $\nu \in Y$, the Fourier--Bohr coefficient
  $a_{\chi}^\cA(\nu)$ exists and 
  \[
  a_{\chi}^\cA(\nu)=0 \,.
  \]
  \item[(e)] $X \subseteq \Bap_{\cA}(G)$.
  \item[(f)] For every $\mu \in X$ and $\nu \in Y$ the following twisted Eberlein convolutions exist and
\[
\ltebe \mu, \nu \rtebeA=\ltebe \nu, \mu \rtebeA=0 \,.
\]
\item[(g)] For all $\mu \in X,\nu \in Y$ and $a,b \in \CC$  the autocorrelation $\gamma_{\omega}$ of $\omega=a\mu+b\nu$ exists with respect to $\cA$ and
\begin{align*}
  \left(\widehat{\gamma_{\omega}}\right)_{\mathsf{pp}} & = |a|^2 \widehat{\gamma} \\
  \left(\widehat{\gamma_{\omega}}\right)_{\mathsf{c}} & = |b|^2 \widehat{\eta} \,.
\end{align*}
\end{itemize}
\end{theorem}
\begin{proof}
By \cite[Thm.~5]{BL}, there exists sets $X_1 \subseteq \XX, Y_1 \subseteq \YY$ of full measure such that, for all $\mu \in X_1$, the autocorrelation of $\mu$ with respect to $\cA$ is $\gamma$ and for all $\nu \in Y_1$, the autocorrelation of $\nu$ with respect to $\cA$ is $\eta$. Next, since $(\XX,m, G)$ has pure point diffraction spectrum, the set $X_2= \XX \cap \Bap_{\cA}(G)$ has full measure in $\XX$ by \cite[Cor.~6.11]{LSS} or \cite[Thm.~3.7]{LSS2}. Therefore
\[
X:= X_1 \cap X_2= X_1 \cap \Bap_{\cA}(G)
\] has full measure in $\XX$.

Now, for each $\chi \in \widehat{G}$, by \cite[Thm.~5]{Len} there exists a set $Y_\chi \in \YY$ of full measure such that $a_\chi^{\cA}(\nu)=0$ for all $\nu \in Y_\chi$. We would like to chose $Y_2:= \cap_{\chi \in \widehat{G}} Y_{\chi}$, but we do not know that this has full measure. We get around this issue by restricting to  $\mathbb{B}$ which is a countable set.

Let
\[
  Y := Y_1 \cap (\cap_{\chi \in \mathbb{B}} Y_{\chi}) \,.
\]
The definition of $Y$ implies that for all $\chi \in B$ and $\nu \in Y$ the Fourier--Bohr coefficient $a_\chi^{\cA}(\nu)$ exists and
\[
a_\chi^{\cA}(\nu) =0 \,.
\]
Moreover, for all $\mu \in X$ we have by \cite[Thm.~3.36]{LSS}
\[
\FB_{\cA}(\mu)=\mathbb{B} \,.
\]
Then \textbf{ (a)-(e)} hold and \textbf{(f)-(g)} follow from Cor.~\ref{prop 6.9}.

\end{proof}

\medskip
Under the assumptions of Thm.~\ref{thm main 2} , if furthermore $(\XX, G)$ is uniquely ergodic then $\gamma$ is the autocorrelation of all
$\omega \in \XX$ \cite[Thm.~5]{BL}. Moreover, the continuity of the eigenfunctions imply in this case that the CPP holds \cite[Thm.~5]{Len}.
In fact, in this case, all measures in $\XX$ must be Weyl and hence Besicovitch almost periodic \cite[Thm.~6.15]{LSS2}.

Therefore, by repeating the proof of Thm.~\ref{thm main 2} we get:

\medskip

\begin{prop} Let $(\XX, G, m)$ and $(\YY, G, n)$ be two ergodic (TMDS), and let $\gamma$ and $\eta$ be the autocorrelations of $(\XX,m)$ and $(\YY,n)$, respectively. Let $\cA$ be any van Hove sequence along which the ergodic theorem holds.

Assume that $\widehat{\gamma}$ is a pure point measure, that $\widehat{\eta}$ is a continuous measure. If $\XX$ is uniquely ergodic and has continuous eigenfunctions, then, there exists a set $Y \subseteq \YY$ with the following properties:
\begin{itemize}
  \item[(a)] The set $Y$ satisfies $n(Y)=1$.
  \item[(b)] For all $\mu \in \XX$ the autocorrelation of $\mu$ with respect to $\cA$ is $\gamma$.
  \item[(c)] For all $\nu \in Y$ the autocorrelation of $\nu$ with respect to $\cA$ is $\eta$.
  \item[(d)] For every $\mu \in \XX$ and $\nu \in Y$ we have
\[
\ltebe \mu, \nu \rtebeA=\ltebe \mu, \nu \rtebeA=0 \,.
\]
\item[(e)] For all $\mu \in \XX,\nu \in Y$ and $a,b \in \CC$  the autocorrelation $\gamma_{\omega}$ of $\omega=a\mu+b\nu$ exists with respect to $\cA$ and
\begin{align*}
  \left(\widehat{\gamma_{\omega}}\right)_{\mathsf{pp}} & = |a|^2 \widehat{\gamma} \\
  \left(\widehat{\gamma_{\omega}}\right)_{\mathsf{c}} & = |b|^2 \widehat{\eta} \,.
\end{align*}
\end{itemize}
\qed
\end{prop}

\medskip

Now, by combining Thm.~\ref{thm main 2} with \cite[Thm.~4.1]{JBA},  we can prove that for an ergodic dynamical system of translation bounded measures the generalized Eberlein decomposition exist almost surely.

\begin{theorem}\label{thm last} Let $(\XX, G, m)$ be an ergodic dynamical system of translation bounded measures with autocorrelation $\gamma$, and $\cA$ a van Hove sequence along which the ergodic theorem holds, Then, there exists two two ergodic dynamical systems of translation bounded measures $(\XX_{\mathsf{pp}}, G , m_{\mathsf{pp}}), (\XX_{\mathsf{c}}, G , m_{\mathsf{c}})$, two Borel factor mappings
\begin{align*}
  \pi_{\mathsf{pp}} &: \XX \to \XX_{\mathsf{pp}} \\
  \pi_{\mathsf{c}} &: \XX \to \XX_{\mathsf{c}} \\
\end{align*}
and some set $\XX' \subseteq \XX$ of full measure such that
\begin{itemize}
  \item[(a)] The diffraction of $(\XX_{\mathsf{pp}}, G , m_{\mathsf{pp}})$ is $(\widehat{\gamma})_{\mathsf{pp}}$.
  \item[(b)] The diffraction of $(\XX_{\mathsf{c}}, G , m_{\mathsf{c}})$ is $(\widehat{\gamma})_{\mathsf{c}}$.
  \item[(c)] For all $\omega \in \XX'$
  \[
  \omega=\pi_{\mathsf{pp}}(\omega)+\pi_{\mathsf{c}}(\omega) \,.
  \]
 is a generalized Eberlein decomposition of $\omega$.
 \end{itemize}
\end{theorem}
\begin{proof}
By \cite[Thm.~4.1]{JBA} there exists some there exists two ergodic dynamical systems of translation bounded measures $(\XX_{\mathsf{pp}}, G , m_{\mathsf{pp}}), (\XX_{\mathsf{c}}, G , m_{\mathsf{c}})$, two Borel factor mappings
\begin{align*}
  \pi_{\mathsf{pp}} &: \XX \to \XX_{\mathsf{pp}} \\
  \pi_{\mathsf{c}} &: \XX \to \XX_{\mathsf{c}} \\
\end{align*}
and a set $Z \subseteq \XX$ of full measure such that
\begin{itemize}
  \item{}The diffraction of $(\XX_{\mathsf{pp}}, G , m_{\mathsf{pp}})$ is $(\widehat{\gamma})_{\mathsf{pp}}$.
  \item{}The diffraction of $(\XX_{\mathsf{c}}, G , m_{\mathsf{c}})$ is $(\widehat{\gamma})_{\mathsf{c}}$.
  \item{} For all $\omega \in Z$ we have
\[
\omega=\pi_{\mathsf{pp}}(\omega) + \pi_{\mathsf{c}}(\omega) \,.
\]
\end{itemize}

Next, let $X \subseteq \XX_{\mathsf{pp}}$ and $Y \subseteq \YY_{\mathsf{pp}}$ be the sets of full measure given by Thm.~\ref{thm main 2}. Since $\pi_{\mathsf{pp}}$ and $\pi_{\mathsf{c}}$ are Borel factor maps, the set $\pi_{\mathsf{pp}}^{-1}(X), \pi_{\mathsf{c}}^{-1}(Y)$ have full measure in $\XX$ and hence so does
\[
\XX' = Z \cap \left(\pi_{\mathsf{pp}}^{-1}(X)\right) \cap \left(\pi_{\mathsf{c}}^{-1}(Y) \right) \,.
\]

The claims follow immediately.
\end{proof}

\subsection*{Acknowledgments} We are grateful to Michael Baake, Daniel Lenz and Timo Spindeler for many insightful discussions which inspired this manuscript. The work was supported by NSERC with grant  2020-00038. We are greatly thankful for all the support.


\begin{thebibliography}{99}


\bibitem{ARMA1} L.~N ~Argabright and J. ~Gil ~de ~Lamadrid,  Fourier analysis of
unbounded measures on locally compact abelian groups, \textit{Memoirs Amer. Math. Soc.}, \textbf{145}, AMS, Providence, RI  (1974).


\bibitem{JBA}
J.-B.~Aujogue,
Pure point/continuous decomposition of translation-bounded measures
and diffraction,
\textit{Ergod.\ Th.\ \& Dynam.\ Syst.} \textbf{40} (2020) 309--352;
\texttt{arXiv:1510.06381}.



\bibitem{BFG}
M.~Baake, N.P.~Frank and U.~Grimm,
Three variations on a theme by Fibonacci,
\textit{Stoch.\ Dyn.} \textbf{21} (2021) 2140001:1--23;
\texttt{arXiv:1910.00988}.

\bibitem{BFGR}
M.~Baake, N.P.~Frank, U.~Grimm and E.A.~Robinson,
Geometric properties of a binary non-Pisot inflationand absence of absolutely continuous diffraction,
\textit{Studia Mathematica} \textbf{247} (2019), 109--154;
\texttt{arXiv:1706.03976}.

\bibitem{BaGa}
M.~Baake and F.~G\"{a}hler,
Pair correlations of aperiodic inflation rules via renormalisation: Some interesting examples,
\textit{Topol. \& Appl.} \textbf{205} (2016) 4--27;
\texttt{arXiv:1511.00885}.

\bibitem{BGM}
M.~Baake, F.~G\"{a}hler and N.~Ma\~{n}ibo,
Renormalisation of pair correlation measures for primitive
inflation rules and absence of absolutely continuous diffraction,
\textit{Commun.\ Math.\ Phys.} \textbf{370} (2019) 591--635;
\texttt{arXiv:1805.09650}.


\bibitem{TAO}
M.~Baake and U.~Grimm,
\textit{Aperiodic Order. Vol. 1: A Mathematical Invitation},
Cambridge University Press, Cambridge (2013).

\bibitem{TAO2}
M.~Baake and U.~Grimm (eds.),
\textit{Aperiodic Order. Vol. 2: Crystallography and
Almost Periodicity},
Cambridge University Press, Cambridge (2017).

\bibitem{BG}
M.~ Baake and U.~ Grimm,
Fourier transform of Rauzy fractals and point spectrum of
1D Pisot inflation tilings,
\textit{Docum. Math.} \textbf{25} (2020) 2303--2337;
\texttt{arXiv:1907.11012}.

\bibitem{BGrM}
M.~ Baake, U.~ Grimm and N.~Ma\~{n}ibo,
Spectral analysis of a family of binary inflation rules,
Lett. Math. Phys. 108 (2018) 1783--1805;
\texttt{arXiv:1709.09083}.

\bibitem{BL}
M.~Baake and D.~Lenz,
Dynamical systems on translation bounded
measures:\ Pure point dynamical and diffraction spectra,
\textit{Ergod.\ Th.\ \& Dynam.\ Syst.} \textbf{24} (2004) 1867--1893;
\newline
\texttt{arXiv:math.DS/0302231}.


\bibitem{BM}
M.~Baake and R.V.~Moody,
{Weighted Dirac combs with pure point diffraction},
\textit{J.\ Reine Angew.\ Math.\ (Crelle)}
\textbf{573} (2004) 61--94;
\texttt{arXiv:math.MG/0203030}.

\bibitem{BSS}
M.~Baake, T.~Spindeler and N.~Strungaru,
Diffraction of compatible random substitutions in one dimension,
\textit{Indag.\ Math.} \textbf{29} (2018) 1031--1071;
\texttt{arXiv:1712.00323}.


\bibitem{BS}
M. Baake and N. Strungaru, Eberlein decomposition for PV inflation systems, \emph{Lett. Math. Phys.} \textbf{111}(2021), 21 pp.;
\texttt{arXiv:2005.06888 }


\bibitem{BF}
C. ~Berg and G. ~Forst, \textit{ Potential Theory on Locally Compact
Abelian Groups}, Springer, Berlin (1975).


\bibitem{EBE}
W.F. ~Eberlein, Abstract ergodic theorems and weak almost
periodic functions, \textit{Trans. Amer. Math. Soc.}, \textbf{67}(1949) 217--24.



\bibitem{Gou2}
J.-B. Gou\'{e}r\'{e}, Diffraction and Palm measure of
point processes,\textit{C. R. Acad. Sci. Paris} \textbf{336}(2003), 57--62;
\texttt{arXiv:math/0208064}.

\bibitem{Gou}
J.-B. Gou\'{e}r\'{e},Quasicrystals and almost
periodicity,  \textit{Commun. Math. Phys.} \textbf{255}(2005), 655--681;
\texttt{arXiv:math-ph/0212012}.


\bibitem{Hof2}
A.~Hof, Uniform distribution and the projection method, in \textit{ Quasicrystals and Discrete Geometry}, ed. J. Patera,
Fields Institute Monographs \textbf{10} (1988) AMS, Providence, RI, pp.~201--206.

\bibitem{HOF3} A. Hof, On diffraction by aperiodic structures,  \textit{Commun. Math. Phys.} \textbf{169} (1995) 25--43.


\bibitem{ARMA}
J.~Gil de Lamadrid and L.N.~Argabright,
{Almost periodic measures},
\textit{Memoirs Amer.\ Math.\ Soc.} \textbf{85}, no.~428,
AMS, Providence, RI (1990).


\bibitem{LMS}
J.-Y.\ Lee, R.~V.~Moody and  B.~Solomyak, Pure point
dynamical and diffraction spectra,  \textit{Ann. H.\ Poincar\'{e}} \textbf{ 3}
(2002), 1003--1018; \texttt{arxiv:0910.4809}.


\bibitem{Len}
D.~Lenz,
Continuity of eigenfunctions of uniquely ergodic dynamical
systems and intensity of Bragg peaks,
\textit{Commun.\ Math.\ Phys.} \textbf{287} (2009) 225--258;
\texttt{arXiv:math-ph/0608026}.


\bibitem{LSS}
D.~Lenz, T.~Spindeler and N.~Strungaru,
Pure point diffraction and mean, Besicovitch and Weyl almost periodicity,
\textit{preprint} (2020); \texttt{arXiv:2006.10821}.

\bibitem{LSS2}
D.~Lenz, T.~Spindeler and N.~Strungaru,
Pure point spectrum for dynamical systems and mean
almost periodicity, \textit{preprint} (2020);
\texttt{arXiv:2006.10825.}


\bibitem{LSS3}
D. Lenz, T. Spindeler and N. Strungaru, The (twisted) Eberlein convolution of measures, \textit{preprint} (2022);
\texttt{arXiv:2211.06969.}

\bibitem{LS}
D.~Lenz and N.~Strungaru,
{On weakly almost periodic measures},
\textit{Trans.\ Amer.\ Math.\ Soc.} \textbf{371} (2019) 6843--6881;
\texttt{arXiv:1609.08219}.

\bibitem{LS2}
D. Lenz and N. Strungaru, {Pure point spectrum for measurable
dynamical systems on locally compact Abelian groups},\textit{ J. Math. Pures Appl.} \textbf{92}(2009), 323--341;
\texttt{arXiv:0704.2498}.

\bibitem{Man}
N.~Ma\~{n}ibo,
Lyapunov exponents for binary substitutions of constant length,
\textit{J. Math. Phys.} \textbf{58}(2017) 113504:1-9;
\texttt{arXiv:1706.00451}.



\bibitem{Moll}
M.~Moll,
Diffraction of random noble means words,
\textit{J.\ Stat.\ Phys.} \textbf{156} (2014) 1221--1236;
\texttt{arXiv:1404.7411}.

\bibitem{MS}
R. V. Moody and N. Strungaru, { Point sets and dynamical
systems in the autocorrelation topology}, \textit{Canad. Math. Bull.}
\textbf{47}(2004), 82--99.


\bibitem{MoSt}
R.V.~Moody and N.~Strungaru,
{Almost periodic measures and their Fourier transforms},
in \cite{TAO2}, pp.~173--270.


\bibitem{Ped}
G.~K.~Pedersen, \textit{Analysis Now}, Springer, New York (1989);
Revised \ printing (1995).

\bibitem{rud2}
W.~Rudin, {\em Real and Complex Analysis}, McGraw-Hill, New York (1986).

\bibitem{RS}
D.~ Rust and T.~ Spindeler,
Dynamical systems arising from random substitutions,
\textit{Indag.\ Math.} \textbf{29} (2018) 1131--1155;
\texttt{arXiv:1707.09836}.


\bibitem{She}
D. Shechtman, I. Blech, D. Gratias and J.  W. Cahn,
Metallic phase with long-range orientational order
and no translation symmetry,
\textit{Phys.\ Rev.\ Lett.\ }\textbf{53}(1984) 183--185.



\bibitem{SOL2}
B. Solomyak, {Dynamics of self symilar tilings}, \textit{Ergod.\ Th.\ \& Dynam.\ Syst.} \textbf{17}(1997), 695--738.

\bibitem{SOL}
B.~Solomyak, {Spectrum of dynamical systems arising from Delone sets},
in  \textit{Quasicrystals and Discrete Geometry} (Toronto, ON, 1995), ed. J.~Patera, Fields Inst. Monogr., \textbf{10}, AMS, Providence, RI(1998), pp. 265--275.


\bibitem{Timo}
T.~Spindeler,
\textit{Spectral Theory of Compatible Random Inflation Systems},
PhD thesis, Univ.\ Bielefeld (2018);
\texttt{urn:nbn:de:0070-pub-29173839}.



\bibitem{NS2}
N.~Strungaru, Almost periodic measures and long-range
order in Meyer sets, \textit{Discr. Comput. Geom.} \textbf{33}(2005), 483--505.


\bibitem{NS14}
N.~Strungaru,
On weighted Dirac combs supported inside model sets,
\textit{J.\ Phys.\ A:\ Math.\ Theor.} \textbf{47} (2014)
335202:1--19; \texttt{arXiv:1309.7947}.

\bibitem{NS11}
N.~Strungaru,
{Almost periodic pure point measures},
in \cite{TAO2}, pp.~271--342;
\texttt{arXiv:1501.00945}.

\bibitem{NS20}
N.~Strungaru,
{On the Fourier analysis of measures with Meyer set support},
\textit{J.\ Funct.\ Anal.} \textbf{278} (2020) 108404:1--30;
\texttt{arXiv:1807.03815}.



\bibitem{NS21}
N.~Strungaru,
Why do Meyer sets diffract?,
{\it preprint} (2021);
\texttt{arXiv:2101.10513}.

\end{thebibliography}
\end{document}